\documentclass[10pt]{amsart}

\usepackage{graphicx}
\usepackage{color}
\usepackage[inline]{enumitem}
\usepackage{hyperref}
\usepackage{stmaryrd} 
\usepackage{multicol}
\usepackage{tikz-cd}
\usepackage{stix}
\usepackage{dsfont}

\usepackage[a4paper,top=3cm,bottom=2cm,left=2.5cm,right=2.5cm,marginparwidth=1.75cm]{geometry}

\usepackage{ulem}
\usepackage{soul}

\newcommand{\Ad}{\mathrm{Ad}}
\newcommand{\KZ}{\mathrm{KZ}}
\newcommand{\De}{\mathrm{De}}
\newcommand{\Li}{\mathrm{Li}}
\newcommand{\reg}{\mathrm{reg}}
\newcommand{\rev}{\mathrm{rev}}
\newcommand{\wt}{\mathrm{wt}}
\newcommand{\dep}{\mathrm{dp}}

\newcommand{\newaction}{\raisebox{0.6ex}{\scalebox{0.4}{\ \   $\blacktriangleright$ \ }}}

\input cyracc.def
\DeclareFontFamily{U}{russian}{}
\DeclareFontShape{U}{russian}{m}{n}
        { <5><6> wncyr5
        <7><8><9> wncyr7
        <10><10.95><12><14.4><17.28><20.74><24.88> wncyr10 }{}
\DeclareSymbolFont{Russian}{U}{russian}{m}{n}
\DeclareSymbolFontAlphabet{\mathcyr}{Russian}
\makeatletter
\let\@math@cyr\mathcyr
\renewcommand{\mathcyr}[1]{\@math@cyr{\cyracc #1}}
\makeatother
\newcommand{\sh}{\mathcyr{sh}} 

\newtheorem{thm}{Theorem}[section]
\newtheorem{lem}[thm]{Lemma}

\newtheorem{prop}[thm]{Proposition}
\newtheorem{ex}[thm]{Example}

\theoremstyle{definition}
\newtheorem{definition}[thm]{Definition}

{\theoremstyle{definition} \newtheorem{rem}[thm]{Remark}}
{\theoremstyle{definition} \newtheorem{notation}[thm]{Notation}}
{\theoremstyle{definition} }

{\theoremstyle{remark} }

\numberwithin{equation}{section}

\title[$p$-adic MZVs and binomial multiple harmonic sums]{$p$-adic multiple zeta values and binomial multiple harmonic sums}

\author{Hidekazu Furusho}
\author{David Jarossay}

\address{Graduate School of Mathematics, Nagoya University, 
Furo-cho, Chikusa-ku, Nagoya, 464-8602, Japan}
\email{furusho@math.nagoya-u.ac.jp}

\address{De Vinci Higher Education, De Vinci Research Center, Paris, France}
\email{david.jarossay@devinci.fr}

\date{November 28, 2025}

\begin{document}

\begin{abstract}
We present a concise method for deriving an explicit formula for 
$p$-adic multiple zeta values. 
The formula features a variant of multiple harmonic sums, termed binomial multiple harmonic sums.

\end{abstract}

\maketitle
{\footnotesize \tableofcontents}

\setcounter{section}{-1}
\section{Introduction}

Multiple zeta values, which are prominent examples of periods in algebraic geometry, 
are the following real numbers : for $n_{1},\ldots,n_{d} \in \mathbb{N}_{\geq 1}$ with $n_{d}\geq 2$,

\begin{eqnarray} \zeta(n_{1},\ldots,n_{d}) &=& \sum_{0<m_{1}<\ldots<m_{d}} \frac{1}{m_{1}^{n_{1}} \cdots m_{d}^{n_{d}}}
\label{eq:series}
\\ &=& (-1)^r\int_{\Delta} \frac{dt_{1}}{t_{1} - \epsilon_{1}} \wedge \cdots \wedge \frac{dt_{n}}{t_{n} - \epsilon_{n}}  \label{eq:integrals}
\end{eqnarray}
where $n=n_{1} + \cdots + n_{d}$, $(\epsilon_{1},\ldots,\epsilon_{n})=(1,\underbrace{0,\ldots,0}_{n_{1}-1},\ldots,1,\underbrace{0,\ldots,0}_{n_{r}-1})$, and $\Delta = \{(t_{1},\ldots,t_{n}) \in \mathbb{R}^{n} \mid 0 <t_{1}< \cdots <t_{n} < 1\}$. We say that $n$ is the weight of $(n_{1},\ldots,n_{d})$ and that $d$ is the depth of $(n_{1},\ldots,n_{d})$. 
The $p$-adic multiple zeta values $\zeta_{p}^{\KZ}(n_{1},\ldots,n_{r}) \in \mathbb{Q}_{p}$, introduced by the first-named author \cite{F04}
are their $p$-adic analogues in the sense of Coleman integration via equation \eqref{eq:integrals}. 
In this paper we employ the numbers $\zeta_{p}^{\De}(n_{1},\ldots,n_{r}) \in \mathbb{Q}_{p}$, Deligne's variant of $p$-adic multiple zeta values
(consult \cite[\S 2.1]{F07} and \cite[\S 5.28]{DG}).

The definition of $p$-adic multiple zeta values does not come with an explicit formula. 
It was suggested in \cite{DG} to find a formula for $p$-adic multiple zeta values in terms of $p$-adic series, analogous to equation \eqref{eq:series}.

In one variable, it is known that $p$-adic zeta values coincide with special values of the $p$-adic Kubota-Leopoldt zeta function : for all $n\geq 1$, $\zeta_{p}^{\De}(n) = L_p(n, \varpi^{1-n})$. In higher depth, \"Unver \cite{U04, U15} and the second author \cite{J1,J2,J3} found formulas for $p$-adic multiple zeta values with different approaches and results. The second author's approach had several applications, including : proving a conjecture of Akagi, Hirose and Yasuda relating multiple harmonic sums to $p$-adic multiple zeta values ;  discovering and studying the notions of pro-unipotent harmonic actions, adjoint multiple zeta values and multiple harmonic values \cite{J4} ; and non-vanishing results \cite{J5}.

The goal of this paper is to provide what seems to us to be the shortest way to a formula for $p$-adic multiple zeta values in terms of $p$-adic series, in order to make $p$-adic multiple zeta values more accessible to non-expert readers. 

Our approach is a simpler variant of \"{U}nver's approach \cite{U04,U15}, where we have only one inductive equation to consider, and where we do not have to compute inductively overconvergent $p$-adic multiple polylogarithms. The simplification comes from the automorphism $z \mapsto \frac{z}{z-1}$ of $\mathbb{P}^{1} \setminus \{0,1,\infty\}$, which exchanges $1$ and $\infty$. Moreover, we express the inductive equation concisely by using certain combinatorial tools on non commutative formal power series.

Usually, the formulas for $p$-adic multiple zeta values involve multiple harmonic sums :

\begin{notation}
We call {\it multiple harmonic sums} the following rational numbers : for $d  \in \mathbb{N}_{\geq 1}$, $(n_{1},\ldots,n_{d}) \in \mathbb{N}^{d}_{\geq 1}$ and $m \in \mathbb{N}_{\geq 1}$,
\begin{equation*}\label{eq: MHS}
 h_{n_{1},\ldots,n_{d}}(m) := \sum_{0<m_{1}<\cdots < m_{d} < m} \frac{1}{m_{1}^{n_{1}} \cdots m_{d}^{n_{d}}}\in{\mathbb Q} 
 \end{equation*}

These numbers appear in the power series expansions of multiple polylogarithms, see equations \eqref{eq:polylog} and \eqref{eq: polylog MHS 2}.

In this paper, we will use instead the following variant :
\end{notation}

\begin{definition} \label{def bmhs} 
We call {\it binomial multiple harmonic sums} the following numbers : for $d  \in \mathbb{N}_{\geq 1}$, $(n_{1},\ldots,n_{d}) \in \mathbb{N}^{d}_{\geq 1}$ and $m \in \mathbb{N}_{\geq 1}$,
$$h^{B}_{n_{1},\ldots, n_{d}}(m) = (-1)^{m} \sum_{1 \leq k \leq m}\frac{h_{n_{1},\ldots,n_{d-1}}(k)}{k^{n_{d}}} {m-1 \choose m -k}\in{\mathbb Q},$$
An index $(n_{1},\ldots,n_{d})$ can also be interpreted as a word on the alphabet $\{e_{0},e_{1}\}$ by $(n_{1},\ldots,n_{d}) = e_{0}^{n_{d}-1}e_{1}\cdots e_{0}^{n_{1}-1}e_{1}$. Then, we extend the definition of $h^{B}_{w}$ to the case where the rightmost letter of $w$ is $e_{0}$, by : for any $n_{0} \in \mathbb{N}_{\geq 1}$, 
$$ h^{B}_{e_{0}^{n_{d}-1}e_{1}\cdots e_{0}^{n_{1}-1}e_{1}e_{0}^{n_{0}-1}} = \sum_{\substack{0\leq k_{1},\ldots,k_{d} \\ k_{1}+\cdots+k_{d} = n_{0}-1}} \prod_{i=1}^{d} {-n_{i} \choose k_{i}} h^{B}_{ e_0^{n_{d}+k_{d}-1} e_1 \cdots e_0^{n_{1}+k_{1} - 1}e_1 }\in \mathbb{Q}. $$
and to the case of depth 0 : for $(i,n) \in \mathbb{N}^{2}_{\geq 0}$, $h^{B}_{e_{0}^{n}}(i)= \left\{ \begin{array}{ll} 1 & \text{if } (i,n)=(0,0)
\\ 0 & \text{ otherwise } \end{array} \right. $.
\end{definition}

Binomial multiple harmonic sums appear in the power series expansion of multiple polylogarithms to which we have applied the automorphism $z \mapsto \frac{z}{z-1}$ of $X=\mathbb{P}^{1} \setminus \{0,1,\infty\}$, see the formulae \eqref{eq: relation between BMHS and MPL} and \eqref{eq: relation between pBMHS and MPL}.

The main result of this paper is the following formula for $p$-adic multiple zeta values, in terms of binomial multiple harmonic sums :

\begin{thm}\label{thm 0.2}
For all $d \in \mathbb{N}_{\geq 1}$ and $(n_{1},\ldots,n_{d}) \in \mathbb{N}_{\geq 1}^{d}$
,
we have

\begin{multline*} \label{eq:2.2}
\zeta_{p}^{\De}(n_{1},\ldots,n_{d}) 
=(-1)^{d}  \lim_{|m|_{p}\rightarrow 0} 
\sum_{d'=0}^{d}\sum_{\substack{a,b\geq 0 \\ a+b \leq n_{d'}-1}}
\sum_{\substack{0\leq i,j,l \\ i+l+pj = m}} \bigg( h^{B}_{e_0^{n_{d}-1}e_1 \dots e_0^{n_{d'+1}-1}e_1e_{0}^{a}}(i)  
 \cdot 
\sum_{\substack{i_1,\ldots,i_b \geq 0 \\ i_1+\cdots+i_b=l \\ (i_1,p)=\cdots=(i_b,p)=1}}\frac{1}{b!\cdot i_1\cdots i_b} \cdot
\\
\
\sum_{\substack{1\leq r \\ (w_{k})_{k=1}^{r}  \text{$e_1$-segments of }\\ (e_0^{n_{d'}-1-a-b} e_1 \dots e_0^{n_{1}- 1}e_1)^\rev}} 
\frac{(-1)^{\sum_{i=1}^{d'} n_{i} - a-b}}{p^{\sum_{i=1}^{d'}n_{i}-a-b-\sum_{k=1}^r(\wt(w_k)-1)}}
\bigg((-1)^{d'-r}\prod_{k=1}^{r}  (\zeta_{p}^{\De})^{\Ad}(w_{k}) \bigg)\cdot h^{B}_{(e_0^{n_{d'}-1-a-b} e_1 \dots e_0^{n_{1}- 1}e_1)^{\rev} / (w_{k})_k} (j) \bigg) .
\end{multline*}
\end{thm}

Here, for $w$ a word on $\{e_{0},e_{1}\}$, $w^{\rev}$ is the word reversed, i.e. read backwards ; the notions of $e_1$-segment $(w_{k})_{k}$ of a word $w$ and of the contraction $w/(w_{k})_{k}$ are defined in Definition \ref{def: segment} ; and we have $n_0 = 1$. The $p$-adic numbers $ (\zeta_{p}^{\De})^{\Ad}(w_{k})$, defined as {\it adjoint $p$-adic multiple zeta values} in \cite{J4}, are certain degree two polynomial of $p$-adic multiple zeta values (see equations \eqref{eq:adjoint pMZVs 0}, \eqref{eq:adjoint pMZVs}).

The formula of the theorem is a $p$-adic analogue of equation (\ref{eq:series}). It also implies a formula for the first author's notion of $p$-adic multiple zeta values, the numbers $\zeta_{p}^{\KZ}(n_{1},\ldots,n_{d}) \in \mathbb{Q}_{p}$, via Theorem 2.8 of \cite{F07}, which relates the two families of numbers $\zeta_{p}^{\KZ}(w)$ and $\zeta_{p}^{\De}(w)$.

The formula is inductive with respect to the depth of indices. In depth one and two, for example, we obtain the following formulae :

\begin{ex} \label{ex depth1}
For $n \in \mathbb{N}_{\geq 1}$,
we have
\begin{equation*}
\zeta_{p}^{\De}(n) = -\lim_{|m|_{p}\rightarrow 0} \bigg( h^{B}_{e_{0}^{n-1}e_1}(m) +\sum_{0\leq b \leq n-1} \sum_{\substack{l,j\geq 0 \\ l+pj=m}}  \sum_{\substack{i_{1}+\cdots+i_{b}=l \\ \forall j,  p\nmid i_{j} }} \frac{1}{b!\cdot i_{1}\cdots i_{b}} (\frac{-1}{p})^{n-b} h^{B}_{e_{1}e_{0}^{n-1-b}}(j) \bigg)
\end{equation*} 
\end{ex}

\begin{ex} \label{ex depth2}
For $(n_{1},n_{2}) \in \mathbb{N}_{\geq 1}$,
we have
 \begin{multline*}  \zeta_{p}^{\De}(n_{1},n_{2}) = \lim_{|m|_{p}\rightarrow 0} \bigg\{ \sum_{\substack{0\leq a,b \\a+b \leq n_{1}-1}} \sum_{\substack{i,l,j\geq 0 \\ i+l+pj=m}} (\frac{-1}{p})^{n_{1}-b} h^{B}_{e_{0}^{n_{2}-1}e_{1}e_{0}^{a}}(i) \sum_{\substack{i_{1}+\cdots+i_{b}=l \\ \forall j,  p\nmid i_{j} }} \frac{1}{b!\cdot i_{1}\cdots i_{b}}  h^{B}_{e_{0}^{n-1-b}e_{1}}(j)
\\ - 
\sum_{\substack{b\geq 0 \\ b \leq n_{d'}-1}}
\sum_{\substack{0\leq j,l \\ l+pj = m}}
\sum_{\substack{i_1+\cdots+i_b=l \\ (i_1,p)=\cdots=(i_b,p)=1}}\frac{1}{b!\cdot i_1\cdots i_b} \cdot
\sum_{k=0}^{n_{2}-2-b} 
\frac{(-1)^{n_1+n_2-b}}{p^{k+1}}
{-n_{1} \choose n_{2}-1-b-k} \zeta_{p}(n_{1}+n_{2}-1-b-k)
 h^{B}_{e_{1}e_{0}^{k}} (j)
\\ +
\sum_{\substack{b\geq 0 \\ b \leq n_{2}-1}}
\sum_{\substack{0\leq j,l \\ l+pj = m}} 
\sum_{\substack{i_1+\cdots+i_b=l \\ (i_1,p)=\cdots=(i_b,p)=1}}\frac{1}{b!\cdot i_1\cdots i_b} \cdot
\frac{(-1)^{n_{1}}}{p^{n_{2}-1}} 
\cdot h^{B}_{e_{1}e_{0}^{n_{1}-1}e_{1}e_{0}^{n_{2}-1-b}}
\bigg\}
\end{multline*}

\end{ex}

{\it Acknowledgments.} This work was done during the second author's stay in Nagoya University in June and July of 2025. 
The first author has been supported by grants JSPS KAKENHI  
JP21H00969 and JP21H04430.
The second author thanks Nagoya University for hospitality.

\section{$p$-adic multiple zeta values as certain limits of $p$-adic multiple polylogarithms}\label{sec:1}

$p$-adic multiple zeta values in the sense of the first author are certain limits of $p$-adic multiple polylogarithms when $z\rightarrow1$ (equation \eqref{eq:lim1}). We write a variant of that formula for $p$-adic multiple zeta values in the sense of Deligne with overconvergent $p$-adic multiple polylogarithms (Proposition  \ref{prop sec 1}). The overconvergence allows then to relate their power series expansion at $z=0$ with $p$-adic multiple zeta values. The particularity of our approach is a certain way to use the automorphism $z \mapsto \frac{z}{z-1}$ of $\mathbb{P}^{1} \setminus \{0,1,\infty\}$ which exchanges $1$ and $\infty$ to calculate the limit.

\begin{notation} For $R$ a ring, $R\langle \langle e_{0},e_{1}\rangle\rangle$ is the non-commutative $R$-algebra of formal power series over the non commuting variables $e_{0},e_{1}$. An element $P$ of $R\langle \langle e_{0},e_{1}\rangle\rangle$ is denoted by $P =\sum\limits_{w\text{ word on }\{e_{0},e_{1}\}} P_{w} w$ with $P_{w} \in R$ for all $w$. 

For $w$ a word on $\{e_{0},e_{1}\}$, including the empty word denoted by $1$, the {\it weight} of $w$ is its length, which we denote it by $\wt(w)$.
The {\it depth} of $w$ is its degree with respect to $e_1$ which we denote it by $\dep (w)$.
\end{notation}

We fix $\alpha\in\mathbb C_p$ (called a branch parameter) and consider the associated $p$-adic logarithm
$\log_p^\alpha$ , which is the locally analytic group homomorphism 
${\mathbb C_p}^\times\to{\mathbb C}_p$ 
whose Taylor expansion at $1$ coincides with the classical power series, 
and satisfies $\log_p^\alpha(p)=\alpha$.
The construction of the associated $p$-adic multiple polylogarithm $\Li^\alpha_{k_,..,k_m}(z)$
for $k_1,...,k_m$,
which is a function on ${\mathbb P}^1(\mathbb C_p)\setminus\{1,\infty\}$,
together with their generating series $G^\alpha_{0}(z)$ taking values in
$\mathbb C_p\langle\langle e_0,e_1\rangle\rangle$
is introduced in \cite{F04}. 
In detail, for any word $w$ in $e_0$ and $e_1$, 
the coefficient $G^\alpha_0(z)_w$, 
of  $w$ in $G^\alpha_0(z)$, 
is a Coleman function on $\mathbb{P}^1 \setminus \{0,1,\infty\}$ as explained in loc.~cit.
Particularly, for all $r\geq 1$, and $n_{1},\ldots,n_{r} \geq 1$, and for $|z|_{p}<1$, we have 
\begin{equation}
\label{eq:polylog}
\displaystyle 
G^\alpha_{0}(z)_{e_0^{n_{d}-1}e_1\cdots e_0^{n_{1}-1}e_1} = (-1)^d \Li^\alpha_{n_1,\ldots,n_d}(z)
=\displaystyle(-1)^d\sum_{0<m_{1}<\dots < m_{d}} \frac{z^{m_{d}}}{m_{1}^{n_{1}} \cdots m_{d}^{n_{d}}} 
\end{equation}
And we also have $G^\alpha_{0}(z)_{e_0} = \log^\alpha_{p}(z)$. 

In \cite[\S 3.1]{F04},  it is also introduced the $p$-adic KZ associator 
\begin{equation}\label{eq:lim1}
    \Phi_{p}^{\KZ} :=\underset{z\rightarrow 1}{\lim}^\prime G^\alpha_0(z)
\in \mathbb{Q}_{p} \langle \langle e_0,e_1 \rangle \rangle,
\end{equation}
which is the generating series of $p$-adic multiple zeta values.
Here the notation $\lim^{\prime}$ is defined in \cite[Notation 2.12]{F04}.
We have 
$\zeta_{p}^{\KZ}(n_{1},\ldots,n_{d}) =(-1)^d (\Phi_{p}^{\KZ})_{e_0^{n_{d}-1}e_1 \cdots e_0^{n_{1} - 1}e_1}$
for $d \in \mathbb{N}_{\geq 1}$ and $(n_1,\dots,n_d)\geq \mathbb{N}_{\geq 1}^{d}$.

We note that the series $G^\alpha_{0}$ depends on the choice of a branch $\alpha$ of the $p$-adic logarithm, while $\Phi_{p}^{\KZ}$ does not depend on this choice. 

We consider the generating series $G_{0}^{\dagger}(z)$ of
overconvergent $p$-adic multiple polylogarithms (cf. ~\cite[\S 19.6]{Del89} and \cite[\S 2.1]{F07} ).
The series takes values in $\mathbb C_p\langle\langle e_0,e_1\rangle\rangle$
and is defined for $z$ in $\mathbb{P}_{\mathbb C_p}^{1,an} \setminus D(1,1)$
where $D(1,1)=\{z \in \mathbb{C}_{p}\text{ }|\text{ }|z-1|_{p} < 1\}$ 
denotes the open disk of center $1$ and radius $1$.
For any word $w$, $G_{0}^{\dagger}(z)_w$ is an overconvergent analytic function on $\mathbb{P}_{\mathbb C_p}^{1,an} \setminus D(1,1)$. 
Let $\Phi_{p}^{\De} \in \mathbb{Q}_{p} \langle \langle e_0,e_1 \rangle \rangle$ be the series in \cite[Definition 2.7]{F07}
that generates Deligne's $p$-adic multiple zeta values, as defined in \cite[\S5.26]{DG}; 
one has $\zeta_{p}^{\De}(n_{1},\ldots,n_{d}) = (-1)^d(\Phi_{p}^{\De})_{e_0^{n_{d}-1}e_1 \cdots e_0^{n_{1} - 1}e_1}$
for all $d\in \mathbb{N}_{\geq 1}$, and $(n_{1},\ldots,n_{d}) \in \mathbb{N}_{\geq 1}^{d}$. 
The series $G_{0}^{\dagger}$ is given by the following formula (\cite[Theorem 2.14]{F07}) :

\begin{equation} 
G_{0}^{\dagger}(e_0,e_1)(z) = G^\alpha_{0}(e_0,e_1)(z) \cdot G^\alpha_{0}\left( \frac{e_0}{p} , \Phi_p^{\De}(e_0,e_1)^{-1} \frac{e_1}{p} \Phi_p^{\De}(e_0,e_1)\right) (z^{p})^{-1}  
\label{eq:overconvergent polylog}
\end{equation}

Let $e_\infty=-e_0-e_1$. 

\begin{lem} \label{lem 1}
We have $\Phi_{p}^{\mathrm{De}}(e_{0},e_{1}) = G_{0}^{\dagger}(e_{0},e_{\infty})(\infty)$.
\end{lem}

We give two proofs of this lemma.

\begin{proof} Let $\Pi$ be the unipotent de Rham fundamental groupoid of  $(\mathbb{P}^{1} \setminus \{0,1,\infty\} )/ \mathbb{Q}_{p}$ (\cite[\S10]{Del89}). For $x,y$ two base points of $\Pi$ we denote by ${}_y 1 _{x}$ the canonical de Rham path from $x$ to $y$ (\cite[\S12]{Del89}). For $x \in \{0,1,\infty\}$ and $\vec{v} \in \mathbb{Q}_{p}^{\ast}$, we denote by $\vec{v}_{x}$ the tangential base point $\vec{v}$ at $x$ 
(\cite[Definition 2.7]{F07}).

We have, by definition, $\Phi_{p}^{\mathrm{De}} = \phi_{p} (_{\vec{1}_{1}} 1_{\vec{1}_{0}} )$ where $\phi_{p}$ is the crystalline Frobenius automorphism \cite[Definition 2.7]{F07}. The automorphism $z \mapsto \frac{z}{z-1}$ of $\mathbb{P}^{1} \setminus \{0,1,\infty\}$ induces an automorphism $(z\mapsto \frac{z}{z-1})_{\ast}$ of $\Pi$, which commutes with the crystalline Frobenius. Thus we have 

$$ (z\mapsto \frac{z}{z-1})_{\ast}\phi_{p} (_{\vec{1}_{1}} 1_{\vec{1}_{0}} ) = \phi_{p} \Big((z\mapsto \frac{z}{z-1})_{\ast} (_{\vec{1}_{1}} 1_{\vec{1}_{0}}) \Big)  $$

The left hand side above is $\Phi_{p}^{\mathrm{De}}(e_{0},e_{\infty})$. The right-hand side is $\phi_{p} (_{\vec{1}_{\infty}} 1_{\vec{1}_{0}} ) =G_{0}^{\dagger}(\infty)$. Thus we have $\Phi_{p}^{\mathrm{De}}(e_{0},e_{\infty}) = G_{0}^{\dagger}(e_{0},e_{1})(\infty)$. By applying the algebra involutive automorphism of $\mathbb{Q}_{p} \langle\langle e_{0},e_{1}\rangle\rangle$ which maps $(e_{0},e_{1}) \mapsto (e_{0},e_{\infty})$, we deduce the result.
\end{proof}

Below we provide another proof based on the techniques of \cite{F04, F07}.

\begin{proof}
By equation \eqref{eq:overconvergent polylog} we have:
$$ G_{0}^\dag(e_0,e_\infty)(\infty)
   = \underset{z\rightarrow \infty}{\lim}^\prime  G_{0}^\dag(e_0,e_\infty)(z) = \underset{z\rightarrow \infty}{\lim}^\prime 
     G^\alpha_{0}(e_0,e_\infty)(z) \cdot G^\alpha_{0}\left( \frac{e_0}{p} , \Phi_p^{\De}(e_0,e_\infty)^{-1} \frac{e_\infty}{p} \Phi_p^{\De}(e_0,e_\infty)\right)(z^p)^{-1} $$
By the variable change $z=\frac{y}{y-1}$, which exchanges $1$ and $\infty$, we deduce
$$ G_{0}^\dag(e_0,e_\infty)(\infty) = \underset{y\rightarrow 1}{\lim}^\prime 
     G^\alpha_{0}(e_0,e_\infty)\Big(\frac{y}{y-1}\Big) \cdot  G^\alpha_{0}\left( \frac{e_0}{p} , \Phi_p^{\De}(e_0,e_\infty)^{-1} \frac{e_\infty}{p} \Phi_p^{\De}(e_0,e_\infty)\right)\Big((\frac{y}{y-1})^p\Big)^{-1} $$

By \cite[\S 2.2]{F22}, we also have
\begin{equation}\label{eq: relation for G0 at 0 and infinity}
    G^\alpha_0(e_0,e_1)(z)=G^\alpha_0(e_0,e_\infty)(\frac{z}{z-1}).
\end{equation}

Of course, this equality remains true if we replace $e_{0},e_{1},e_{\infty}$ by any series $A,B,C$ in $\mathbb Q\langle\langle e_0,e_1\rangle\rangle$
without constant terms such that $A+B+C=0$. We deduce :

$$ G_{0}^\dag(e_0,e_\infty)(\infty) = \underset{y\rightarrow 1}{\lim}^\prime 
     G^\alpha_{0}(e_0,e_1)(y) \cdot  G^\alpha_{0}\left( \frac{e_0}{p} , - \frac{e_0}{p} - \Phi_p^{\De}(e_0,e_\infty)^{-1} \frac{e_1}{p} \Phi_p^{\De}(e_0,e_\infty)\right)\Big(\frac{y^p}{y^{p} - (y-1)^p}\Big)^{-1} $$

We also have 
\begin{equation} \label{eq:residue sum}e_0+\Phi_{p}^{\De}(e_0,e_1)^{-1}e_1\Phi_{p}^{\De}(e_0,e_1)+\Phi_{p}^{\De}(e_0,e_\infty)^{-1}e_\infty\Phi_{p}^{\De}(e_0,e_\infty)=0
\end{equation}

Indeed $\Phi_{p}^{\De}(e_0,e_1)$ is in Drinfeld's \cite{Dr} group $\mathrm{GRT}(\mathbb Q_p)$ by \cite{F07, U13}, which implies this equality \cite[Proposition 5.6]{Dr}.
Thus we deduce :
$$ G_{0}^\dag(e_0,e_\infty)(\infty) = \underset{y\rightarrow 1}{\lim}^\prime 
     G^\alpha_{0}(e_0,e_1)(y) \cdot  G^\alpha_{0}\left( \frac{e_0}{p} , \Phi_p^{\De}(e_0,e_1)^{-1} \frac{e_1}{p} \Phi_p^{\De}(e_0,e_1)\right)\Big(\frac{y^p}{y^{p} - (y-1)^p}\Big)^{-1}. $$

By equation  \eqref{eq:lim1} 
and
$\displaystyle \Phi_p^\De(e_0,e_1)=\Phi_p^\KZ(e_0,e_1)\Phi_p^\KZ( \frac{e_0}{p} , \Phi_p^{\De}(e_0,e_1)^{-1} \frac{e_1}{p} \Phi_p^{\De}(e_0,e_1))^{-1}$ shown in \cite[Theorem 2.8]{F07}, we obtain the result.

\end{proof}

\begin{lem}\label{lem: new formula for Gdag}
We have
\begin{equation*}
G_0^\dag(e_0,e_\infty)(z)=G^\alpha_{0}(e_0,e_1)\Big(\frac{z}{z-1}\Big)\cdot  G^\alpha_{0}\left( \frac{e_0}{p} , \Phi_p^{\De}(e_0,e_1)^{-1} \frac{e_1}{p} \Phi_p^{\De}(e_0,e_1)\right) \bigg(\frac{z^{p}}{z^{p}-1}\bigg)^{-1}.
\end{equation*}
\end{lem}
\begin{proof} We apply the automorphism $(e_{0},e_{1})\mapsto (e_{0},e_{\infty})$ to equation (\ref{eq:overconvergent polylog}). By equation \eqref{eq: relation for G0 at 0 and infinity} and equation \eqref{eq:residue sum}, the right-hand side becomes
\begin{align*}
G^\alpha_{0}(e_0,e_\infty)(z) &\cdot G^\alpha_{0}\left( \frac{e_0}{p} , \Phi_p^{\De}(\frac{e_0}{p},e_\infty)^{-1} \frac{e_\infty}{p} \Phi_p^{\De}(e_0,e_\infty)\right) (z^{p})^{-1}  \\
&= G^\alpha_{0}(e_0,e_1)(\frac{z}{z-1})\cdot  G^\alpha_{0}\left( \frac{e_0}{p} , \Phi_p^{\De}(e_0,e_1)^{-1} \frac{e_1}{p} \Phi_p^{\De}(e_0,e_1)\right) \bigg(\frac{z^{p}}{z^{p}-1}\bigg)^{-1}.
\end{align*}
Whence we obtain the claim.
\end{proof}

Let $G_{0}^{\alpha, \reg}$
be the regularization of $G^\alpha_{0}$ in $z=0$. It is defined by $G^\alpha_{0}(e_0,e_1)(z) = G_{0}^{{\alpha, \reg}}(e_0,e_1)(z) \exp\{\log^\alpha_{p}(z)e_0\}$. 
We have $G^\alpha_{0} (z)_w = G_{0}^{\alpha, \reg} (z)_w$ for all words $w$ of the form $w' e_1$ where $w'$ is a word, $G_{0}^{\alpha, \reg}(z)_{e_0}  = 0$ and $G_{0}^{\alpha, \reg}$ 
is group-like, i.e. its coefficients
satisfy the shuffle equation.

\begin{notation} Let $\Li_{1}^{(p)}(z) = \Li^\alpha_{1}(z) - \frac{1}{p}\Li^\alpha_{1}(z^{p})$ for $z \in \mathbb{C}_{p} \setminus \{1\}$.
In the region this function is independent of $\alpha$.
\end{notation}

\begin{prop} \label{prop sec 1}
We have 
\begin{multline} \label{eq:lim2}
\Phi_{p}^{\De}(e_0,e_1) = 
\\ G_{0}^{\alpha, \reg}(e_0,e_1)\Big(\frac{z}{z-1}\Big)  \cdot \exp \bigg(\Li_{1}^{(p)}(z) e_{0} \bigg) \cdot G_{0}^{\alpha, \reg}\left( \frac{e_0}{p} , \Phi_p^{\De}(e_0,e_1)^{-1} \frac{e_1}{p} \Phi_p^{\De}(e_0,e_1)\right)\Big(\frac{z^{p}}{z^{p}-1}\Big)^{-1}\Bigm|_{z=\infty}.
\end{multline}
\end{prop}

\begin{proof} 

By Lemma \ref{lem 1} and Lemma \ref{lem: new formula for Gdag}, we deduce
\begin{equation} \label{eq:2} 
\Phi_{p}^{\De}(e_{0},e_{1}) = G^\alpha_{0}(e_0,e_1)(\frac{z}{z-1})\cdot  G^\alpha_{0}\left( \frac{e_0}{p} , \Phi_p^{\De}(e_0,e_1)^{-1} \frac{e_1}{p} \Phi_p^{\De}(e_0,e_1)\right) \bigg(\frac{z^{p}}{z^{p}-1}\bigg)^{-1} \Bigm|_{z=\infty}
\end{equation}
By definition, we have
$$ G^\alpha_{0}(e_0,e_1)(z) =  G_{0}^{\alpha, \reg}(e_0,e_1)(z) \cdot\exp( e_0 \log^\alpha_{p}(z)) $$
thus 
$$ G^\alpha_{0}(e_0,e_1)(\frac{z}{z-1}) =  G_{0}^{\alpha, \reg}(e_0,e_1)(\frac{z}{z-1}) \cdot\exp( e_0 \log^\alpha_{p}(\frac{z}{z-1})) $$
and, since $\log^\alpha_{p}(z^{p}) = p \log^\alpha_{p}(z)$,   
\begin{align*}
G^\alpha_{0}&\left( \frac{e_0}{p} ,  \Phi_p^{\De}(e_0,e_1)^{-1} \frac{e_1}{p} \Phi_p^{\De}(e_0,e_1)\right) (\frac{z^{p}}{z^{p}-1})^{-1} 
\\ &= \bigg( G_{0}^{\alpha, \reg}\left( \frac{e_0}{p} , \Phi_p^{\De}(e_0,e_1)^{-1} \frac{e_1}{p} \Phi_p^{\De}(e_0,e_1)\right) (\frac{z^{p}}{z^{p}-1}) \exp( \frac{e_0}{p} \log^\alpha_{p} (\frac{z^{p}}{z^{p}-1})) \bigg)^{-1} 
\\ 
&= \exp ( \frac{e_0}{p} \log^\alpha_{p}(\frac{z^{p}}{z^{p}-1}))^{-1} G_{0}^{\alpha, \reg}\left( \frac{e_0}{p} , \Phi_p^{\De}(e_0,e_1)^{-1} \frac{e_1}{p} \Phi_p^{\De}(e_0,e_1)\right) (\frac{z^{p}}{z^{p}-1})^{-1} 
\end{align*}
and we have 
\begin{eqnarray*} \exp( e_0 \log^\alpha_{p}(\frac{z}{z-1})) \exp(\frac{e_0}{p} \log^\alpha_{p}(\frac{z^{p}}{z^{p}-1}))^{-1} &=& 
\exp \bigg( \Big( -\log_{p}^{\alpha}(z-1) + \frac{1}{p}\log_{p}^{\alpha}(z^{p}-1) \Big) e_{0}\bigg)
\\ &=& \exp \bigg( (\Li^\alpha_{1}(z) - \frac{1}{p}\Li^\alpha_{1}(z^{p})) e_{0} \bigg)
\end{eqnarray*}
because $\log_{p}^{\alpha}(z) = \log_{p}^{\alpha}(-z)$, and $\Li_{1}(z) = - \log_{p}^{\alpha}(1-z)$.
Whence our claim follows.
\end{proof}


\section{Combinatorics of non-commutative formal power series in the limit equation}\label{sec:2}

In this section we formalize the combinatorics of Proposition \ref{prop sec 1}.

\begin{definition} \label{op reg} 
For $P \in \mathbb{Q}_{p}\langle\langle e_{0},e_{1}\rangle\rangle$ and $L \in \mathbb{Q}_{p}[[z]]\langle \langle e_{0},e_{1}\rangle\rangle^\times$ and $L_{1} \in \mathbb{Q}_{p}[[z]]\langle \langle e_{0},e_{1}\rangle\rangle$, with $L_{e_{0}^{n}}=0$ for all $n \in \mathbb{N}_{\geq 1}$, let 
$$P \newaction (L,L_{1})(z) = L(e_{0},e_{1})(\frac{z}{z-1})\cdot \exp(L_{1}(z)e_{0}) \cdot L\left(\frac{e_{0}}{p},\ \frac{P^{-1}e_{1}P}{p}\right)(\frac{z^{p}}{z^{p}-1})^{-1}. $$
\end{definition}

With Definition \ref{op reg}, Proposition \ref{prop sec 1} can be reformulated as 

\begin{equation} \label{eq: part 1} \Phi_{p}^{\De} = \Big( \Phi_{p}^{\De}  \text{ } \newaction \text{ }  (G_{0}^{\alpha, \reg},\Li_{1}^{(p)})(z)\Big)|_{z=\infty}.
\end{equation}

The main result of this section is a formula for the coefficients of $P \newaction (L,L_{1})$, in function of $P$, $L$ and $L_{1}$, at any word $w=e_{0}^{n_{d}-1}e_{1}\cdots e_{0}^{n_{1}-1}e_{1}$ (Proposition \ref{prop sec 2}). 

Let $\mathcal{O}^{\sh}$ be the Hopf shuffle algebra over the alphabet $\{e_{0},e_{1}\}$ over $\mathbb{Q}$.
It means the $\mathbb{Q}$-vector space $\mathbb{Q}\langle e_{0},e_{1}\rangle$ freely generated by words on $\{e_{0},e_{1}\}$, with the shuffle product $\sh$, the deconcatenation coproduct, the antipode $S : w \mapsto (-1)^{\mathrm{wt}(w)} w^{\rev}$, where $w^{\rev}$ is the word $w$ read backwards and the counity which maps an element to the coefficient of the constant word. The weight of words defines a grading on $\mathcal{O}^{\sh}$.

The algebra $\mathbb{Q}_{p}\langle \langle e_{0},e_{1}\rangle\rangle$ has a Hopf algebra structure defined as the completion (with respect to the weight grading) of the dual of $\mathcal{O}^{\sh} \otimes \mathbb{Q}_{p}$, by the pairing $(P,w) \mapsto P_{w}$. In particular, if $P\in \mathbb{Q}_{p}\langle \langle e_{0},e_{1}\rangle\rangle$ is grouplike, we have $P^{-1} = \hat{S^{\vee}}(P)$ where $\hat{S^{\vee}}(P) = \sum\limits_{\substack{w \text{ word on}\\ \{e_{0},e_{1}\}}} (-1)^{\wt(w)} P_w w^{\rev}$.

The next definition is a particular case of \cite[Definition 1.2.2]{J2}. 

\begin{definition}\label{def: segment}
Let $w$ be a word on $\{e_0,e_1\}$ and let $(w_{i})_{i}=(w_{1},\ldots,w_{r})$ be an {\it $e_1$-segment of $w$}, that means it is 
a sequence of subwords of $w$ such that $w$ is of the form $w=a_{0}w_{1}a_{1} \ldots a_{r-1} w_{r} a_{r}$, where the $a_{i}$'s are also words and each $w_{i}$ contains at least one letter $e_1$. 
The {\it contraction of $w$} by $(w_{i})_i$, denoted by $w/(w_{i})_i$ is the word obtained by replacing each $w_{i}$ with $e_1$ in $w$, i.e. $a_{0} e_1 a_{1} \cdots a_{r-1} e_1 a_{r}$.
\end{definition} 

The next lemma and its proof are a particular case of the proof of \cite[Proposition 1.2.3]{J2}. We reproduce it for convenience.

\begin{lem} \label{lem contraction}
Let $A,B \in \mathbb{Q}_{p} \langle \langle e_0,e_1\rangle\rangle$, with $A_{e_0^{n}} = B_{e_0^n} = 0$ for all $n\geq 1$ and $B_\emptyset = 0$. We have, for a word $w$,
$$  A(e_0,B)_w = \sum_{\substack{r\geq 1,  (w_{i})_{i=1}^r
\\ \text{$e_1$-segment of }w}} \bigg( \prod_{i=1}^{r}  B_{ w_{i}} \bigg) A_{w / (w_{i})_i}.
$$
\end{lem}

\begin{proof} The depth 0 terms in $B$ being zero, we have :
\begin{align*} 
A&(e_0,B) =  A_\emptyset 1+ \sum_{\substack{ r\geq 1 \\ n_{r},\ldots,n_{0}\geq 1}} A_{e_0^{n_{r}-1}e_1 \dots e_1 e_0^{n_{0} - 1}}  e_0^{n_{r}-1} B \cdots B e_0^{n_{0}-1} \\
&=  A_\emptyset 1 + \sum_{\substack{ r\geq 1 \\ n_{r},\ldots,n_{0}\geq 1}}  A_{e_0^{n_{r}-1}e_1 \dots e_1 e_0^{n_{0} - 1}}  e_0^{n_{r}-1} \Big(\sum_{w_{r}} B_{w_{r}} w_{r}\Big) \cdots \Big(\sum_{w_{1}}  B_{w_{1}} w_{1}\Big) e_0^{n_{0}-1} \\
&=  A_\emptyset 1 + \sum_{\substack{ d\geq 1 \\ m_{d},\ldots,m_{0}\geq 1}} 
\sum_{\substack{r\geq 1 \\ n_{r},\ldots,n_{0}\geq 1
\\ w_{1},\ldots,w_{r} \text{ such that}
\\ e_0^{n_{r}-1}w_{r} \dots w_{1}e_0^{n_{1}-1} \\ = e_0^{m_{d}-1}e_1 \dots e_1 e_0^{n_{0} - 1}}}
B_{ e_0^{n_{r}-1}e_1 \dots e_1 e_0^{n_{0} - 1}}  \bigg( \prod_{i=1}^{r} 
B_{ w_{i}} \bigg) e_0^{m_{d}-1}e_1 \cdots e_1 e_0^{m_{0} - 1} \\
& = A_\emptyset 1 + \sum_{\substack{ d\geq 1 \\ m_{d},\ldots,m_{0}\geq 1}} 
\sum_{\substack{ r\geq 1 , (w_{i})_{i=1}^{r}\\ \text{$e_1$-segments of }\\ e_0^{m_{d} - 1}e_1 \ldots e_1 e_0^{m_{0} - 1} }}
 A_{(e_0^{m_{d}-1}e_1 \cdots e_1 e_0^{m_{0} - 1})/ (w_{i})_{i}} \bigg( \prod_{i=1}^{r} B_{w_{i}} \bigg) e_0^{m_{d}-1}e_1 \cdots e_1 e_0^{m_{0} - 1} 
\end{align*}
whence the result follows.
\end{proof}

We can reformulate the operation introduced in Definition \ref{op reg} by using the results of this section : 

\begin{notation}
For a power series $H=\sum\limits_{n=0}^{\infty} a_{n}z^n\in \mathbb{Q}_{p}[[z]]$, we denote by $H[z^{n}] = a_{n}$ for all $n \geq 0$.
\end{notation}

\begin{prop} \label{prop sec 2} For $P \in \mathbb{Q}_{p}\langle\langle e_{0},e_{1}\rangle\rangle$ and $L \in \mathbb{Q}_{p}[[z]]\langle \langle e_{0},e_{1}\rangle\rangle^\times$ and $L_{1} \in \mathbb{Q}_{p}[[z]]\langle \langle e_{0},e_{1}\rangle\rangle$, such that $L_{w}=0$ for all words $w$ of the form $w'e_{0}$. Let $n_{0}=1$. Then we have

\begin{multline*}
(P \newaction  (L,L_{1}))_{e_{0}^{n_{d}-1}e_{1}\cdots e_{0}^{n_{1}-1}e_{1}}[z^n] = \sum_{d'=0}^{d}
\sum_{\substack{0\leq i,i',j \\ i+i'+j= n}} \sum_{\substack{0 \leq a,b \\ a+b \leq n_{d'}-1}}
\bigg\{ L_{e_0^{n_{d}-1}e_1 \dots e_0^{n_{d'+1}-1}e_1 e_{0}^{a}}\Big(\frac{z}{z-1}\Big)[z^{i}])  \cdot \frac{L_{1}^{b}}{b!}(z)[z^{j}] \cdot
\\ 
\sum_{\substack{1\leq r \\ (w_{k})_{k=1}^{r}  \text{$e_1$-segments of }\\ (e_0^{n_{d'}-a-b-1} e_1 \dots e_0^{n_{1}- 1}e_1)^\rev}}
\bigg(\prod_{k=1}^{r} (P^{-1}e_1 P)_{w_{k}} \bigg)\cdot
\frac{(-1)^{\sum_{i=1}^{d'}n_{i} - a-b}}{p^{\sum_{i=1}^{d'}n_{i}-a-b-\sum_{k=1}^{r}(\mathrm{wt}(w_{k})-1)}}
L_{(e_0^{n_{d'}-a-b-1} e_1 \dots e_0^{n_{1}- 1}e_1)^{\mathrm{rev}} / (w_{k})_k}\Big(\frac{z^{p}}{z^p-1}\Big)[z^{i'}])
\bigg\}.
\end{multline*}
\end{prop}

\begin{proof} Let $P,Q,R \in \mathbb{Q}_{p}\langle\langle e_0,e_1 \rangle\rangle$. For a word $w$ on $\{e_0,e_1\}$, we have
$\displaystyle (P\cdot Q \cdot R)_w = 
\sum_{\substack{w_{1},w_{2},w_{3}\\ \text{words on }e_0,e_1 \\ \text{s.t. } w_{1}w_{2}w_{3} = w} } P_{w_{1}} Q_{w_{2}} R_{w_{3}}$. We apply this to the formula of Definition \ref{op reg}, with $\displaystyle P= L(\frac{z}{z-1})$, $\displaystyle Q=\exp(L_{1}(z) e_{0})$, and $\displaystyle R= L\left(\frac{e_{0}}{p},\ \frac{P^{-1}e_{1}P}{p}\right)\Big(\frac{z^{p}}{z^{p}-1}\Big)$. Then, we consider the coefficient of $z^{n}$ in the power series expansion at $0$, for $n \in \mathbb{N}$. 

We obtain 

\begin{multline*} (P \newaction  (L,L_{1}))_{e_{0}^{n_{d}-1}e_{1}\cdots e_{0}^{n_{1}-1}e_{1}}[z^{n}] =
\\ \sum_{d'=1}^{d} \sum_{\substack{i,i',j\geq 0 \\ i+i'+j=n}} L\Big(\frac{z}{z-1}\Big)_{e_{0}^{n_{d}-1}e_{1}\cdots e_0^{n_{d'}-1}e_{1}e_{0}^{a}}[z^{i}] \cdot e^{L_{1}(z)e_{0}}_{e_{0}^{b}}[z^{j}] \cdot L\Big(\frac{e_{0}}{p},\frac{P^{-1}e_{1}P}{p}\Big)\Big(\frac{z^{p}}{z^{p}-1}\Big)^{-1}_{e_{0}^{n_{d'-1}-1-a-b} e_{1} \cdots e_{0}^{n_{1}-1}e_{1}}[z^{i'}] .
\end{multline*}

If $A \in \mathbb{Q}_{p}\langle \langle e_{0},e_{1}\rangle\rangle$ is grouplike and $B \in \mathbb{Q}_{p}\langle \langle e_{0},e_{1}\rangle\rangle$ is primitive, then $A(e_{0},B)$ is grouplike and $A(e_0,B)^{-1} = A^{-1}(e_0,B)$. Then we apply Lemma \ref{lem contraction}, 
with $A=L(\frac{e_{0}}{p},\frac{e_{1}}{p})$ and $B = P^{-1}e_{1}P$. Indeed, we have then $A_{e_{0}^{n}}= B_{e_{0}^{n}}=0$ for all $n\geq 1$, and, since $P$ is group like, $P^{-1}e_{1}P$ is primitive and we have $S^{\vee}(P) = -P$. We deduce the result.
\end{proof}


\section{Binomial multiple harmonic sums and the power series expansion of $p$-adic multiple polylogarithms}\label{sec:3} 

In this section we discuss the coefficients of the power series expansion in $z$ appearing in Proposition \ref{prop sec 1}.

\begin{lem} \label{shuffle e0a}

Let $P \in {\mathbb Q}_{p} \langle \langle e_0,e_1 \rangle\rangle$ which 

is group-like and such that $P_{e_0}=0$. Then we have 

$$ 
P_{ e_0^{n_{r}-1} e_1 \cdots e_0^{n_{1} - 1}e_1 e_0^{n_{0}- 1}}  = \sum_{\substack{k_{1},\ldots,k_{r} \geq 0\\ k_{1}+\cdots+k_{r} = n_{0}-1}} \prod_{i=1}^{r} {-n_{i} \choose k_{i}} P_{ e_0^{n_{r}+k_{r}-1} e_1 \cdots e_0^{n_{1}+k_{1} - 1}e_1 }
$$

Here ${-n_{i} \choose k_{i}} = (-1)^{k_{i}} {k_{i} + n_{i} - 1 \choose k_{i}}$ is the coefficient of $X^{k_{i}}$ in the power series $(1+X)^{-n_{i}}$. 
\end{lem}

\begin{proof}
This is a well-known fact, proven easily by induction on $n_{0}$, using the fact that $P_{e_0^{n_{d}-1} e_1 \cdots e_0^{n_{1} - 1}e_1 e_0^{n_{0}- 1}\text{ } \sh\text{ } e_0 } 
= P_{e_0^{n_{d}-1} e_1 \cdots e_0^{n_{1} - 1}e_1 e_0^{n_{0}- 1}}P_{e_0} = 0$.
\end{proof}

By the power series expansion of $p$-adic multiple polylogarithms (\ref{eq:polylog}), we have, for all $d \geq 1$ and $n_{1},\ldots,n_{d} \geq 1$, for all $m\geq 1$,

\begin{eqnarray} \label{eq: polylog MHS 1} \Li^\alpha_{n_1,\ldots,n_d}(z) = \sum_{0< m_{r}} \frac{h_{n_{1},\ldots,n_{d-1}}(m_{d})}{m_{d}^{n_{d}}} z^{m_{d}} .
\\ 
 \label{eq: polylog MHS 2}
    h_{n_1,\dots,n_d}(m)= \left\{\frac{z}{1-z}\Li_{n_1\dots n_d}(z)\right\}[z^m]= \sum_{k=1}^{m-1} \Li_{n_1\dots n_d}(z) [z^k].
\end{eqnarray}

What we need is an analogue for the power series appearing in proposition (\ref{prop sec 1}) :

\begin{lem} \label{lemma z/z-1} 

\begin{itemize} 
\item[(i)] For any word $w$ on $e_{0},e_{1}$, we have :

\begin{equation} \label{eq: relation between BMHS and MPL}
   \Li_{w}\Big(\frac{z}{z-1}\Big)[z^m]=
h^{B}_{w}(m) 
\end{equation}

\item[(ii)] For all $d\geq 1$ and $n_{1},\ldots,n_{d} \geq 1$, for all $m\geq 1$, we have
\begin{equation}
\label{eq: relation between pBMHS and MPL}
  \Li_{n_1,\dots,n_d}\left(\frac{z^{p}}{z^{p}-1}\right)[z^m]=
\left\{ \begin{array}{ll} h^{B}_{n_{1},\ldots,n_{d}}(\frac{m}{p}) & \text{if } p\mid m,
\\ 0 & \text{otherwise.}
\end{array}\right.
\end{equation}

\item[(iii)] For $|z|_{p}<1$, we have 

\begin{equation} \Li_{1}(z) - \frac{1}{p}\Li_{1}(z^{p}) = \sum_{\substack{n>0 \\ p \nmid n}} \frac{z^{n}}{n}
\end{equation}
\end{itemize}
\end{lem}

\begin{proof} (i) This follows from equation \eqref{eq: polylog MHS 1}, Definition \ref{def bmhs} and the fact that for any $(a_{n})_{n\geq 1} \in \mathbb{Q}_{p}^{\mathbb{N}}$, we have $\displaystyle \sum_{n=1}^{+\infty} a_{n} \Big(\frac{z}{z-1}\Big)^{n} = \sum_{n=0}^{+\infty} z^{n} (-1)^{n} \sum_{1\leq k \leq n } a_{k}  {n-1 \choose n-k}$.  

(ii) Similar to (i) using instead $\displaystyle \sum_{n=1}^{+\infty} a_{n} \Big(\frac{z^{p}}{z^{p}-1}\Big)^{n} = \sum_{n=0}^{+\infty} z^{pn} (-1)^{n} \sum_{1\leq k \leq n } a_{k}  {n-1 \choose n-k}$.

(iii) For $|z|_{p}<1$, we have $\displaystyle \Li_{1}(z) = \sum_{n>0} \frac{z^{n}}{n}$, and $\displaystyle \frac{1}{p}\Li_{1}(z^{p}) = \sum_{n>0} \frac{z^{pn}}{pn}$.
\end{proof}

\begin{rem} Since $z \mapsto \frac{z}{z-1}$ is an involution, the formula of Definition \ref{def bmhs} admits the following reciprocal formula : for all $d \in \mathbb{N}_{\geq 1}$, $n_{1},\ldots,n_{d} \in \mathbb{N}_{\geq 1}$ and $m \in \mathbb{N}_{\geq 1}$ : 
$\displaystyle \frac{h_{n_{1},\ldots,n_{d-1}}(m)}{m^{n_{d}}} = (-1)^{m} \sum_{1\leq k\leq m} h^{B}_{n_{1},\ldots,n_{d}}(k) {n-1 \choose n-k}$.
In particular, $\displaystyle m^{n_{d}}(-1)^{m} \sum\limits_{1\leq k\leq m} h^{B}_{n_{1},\ldots,n_{d}}(k) {m-1 \choose m-k}$ does not depend on $n_{d}$.
\end{rem}

\begin{rem} The double shuffle relations and, more generally, properties of multiple polylogarithms (including multiple polylogarithms of several variables) imply equations satisfied by binomial multiple harmonic sums and variants of them, by applying $z \mapsto \frac{z}{z-1}$ and considering coefficients of the power series expansion.
\end{rem}

\section{End of the proof and comments}

We finish the proof of Theorem \ref{thm 0.2}, and of the formulas of Example \ref{ex depth1} and Example \ref{ex depth2}. The proof combines the main results of Sections \ref{sec:1},  \ref{sec:2} and \ref{sec:3}, and the following theorem due to Mahler (see also \cite[Appendix A]{J1} for comments on this theorem) :

\begin{thm}[\cite{M}] \label{prop Mahler}
Let $f$ be an analytic function on $(\mathbb{P}^{1,an} \setminus D(1,1))$, 
with power series expansion $f(z) = \sum\limits_{n\geq 0} a_{n} z^{n}$ with $a_n\in\mathbb Q_p$ at $0$. Then $n \in \mathbb{N}^{\ast} \mapsto a_{n}$ extends to a continuous function $\mathbb{Z}_{p} \rightarrow \mathbb{Q}_{p}$ and 
$$ f(\infty) = - \underset{|n|_{p}\rightarrow 0}{\lim} a_{n} $$
\end{thm}

We also need the definition of adjoint $p$-adic multiple zeta values from \cite{J4} : for all words $w'$ on $\{e_{0},e_{1}\}$ :

\begin{equation}  \label{eq:adjoint pMZVs 0} (\zeta_{p}^{\De})^{\Ad}(w') = (-1)^{\mathrm{dp}(w') - 1} ((\Phi_{p}^{\De})^{-1}e_{1}\Phi_{p}^{\De})_{w'} 
\end{equation}

More explicitly, for a word $w'=e_{0}^{b}e_{1}e_{0}^{n_{d}-1} e_{1} \ldots e_0^{n_1 - 1} e_{1} e_{0}^{a}=(a;n_{1},\ldots,n_{d};b)$, with $d\geq 1$, $n_{1},\ldots,n_{d} \in \mathbb{N}_{\geq 1}$ and $a,b\geq 0$, we have :

\begin{multline}
\label{eq:adjoint pMZVs}
(\zeta_{p}^{\De})^{\Ad}(w') =
\\ \sum\limits_{d'=1}^{d-1} 
\sum\limits_{\substack{0 \leq k_{1},\ldots,k_{d} \\ k_{1}+\cdots+k_{d'}=a \\ k_{d'+1}+\cdots+k_{d}=b}}
(-1)^{n_{d'+1}+\cdots+n_{d}+b}\bigg( \prod\limits_{i=1}^{d} {-n_{i} \choose k_{i}} \bigg) \zeta_{p}^{\De}(n_{1}+k_{1},\ldots,n_{d'}+k_{d'})  \zeta_{p}^{\De}(n_{d}+k_{d},\ldots,n_{d'+1}+k_{d'+1}) 
\\ + \delta_{0}(a)\cdot (-1)^{b+n_{1}+\cdots+n_{d}} 
\cdot\bigg( \prod\limits_{i=1}^{d} {-n_{i} \choose k_{i}} \bigg)\cdot
\sum_{\substack{0\leq k_{1},\ldots,k_{d} \\ k_{1}+\cdots+k_{d}=b}} \zeta^{\De}_{p}(n_{d}+k_{d},\ldots,n_{1}+k_{1})\\
+ \delta_{0}(b)\cdot
\bigg( \prod\limits_{i=1}^{d} {-n_{i} \choose k_{i}} \bigg)\cdot
\sum_{\substack{0\leq k_{1},\ldots,k_{d} \\ k_{1}+\cdots+k_{d}=a}}\zeta^{\De}_{p}(n_{1}+k_{1},\ldots,n_{d}+k_{d})
\end{multline}
and $(\zeta_{p}^{\De})^{\Ad}(e_{0}^{n})=0$ for all $n \in \mathbb{N}_{\geq 1}$, and $(\zeta_{p}^{\De})^{\Ad}(e_{0}^{a}e_{1}e_{0}^{b})=\delta_{0}(a)\delta_{0}(b)$ for all $a,b \in \mathbb{N}_{\geq 0}$. Here $\delta_0(n)=1$ when $n=0$ and 0 when $n\neq 0$.

We can now finish the proof of Theorem \ref{thm 0.2}.

\bigskip
{\it Proof of Theorem \ref{thm 0.2}. }
Combining Proposition \ref{prop sec 2}, Lemma \ref{lemma z/z-1} and equation \eqref{eq:adjoint pMZVs 0}, and the fact that for $(w_{i})_{i=1,\ldots,r}$ an $e_{1}$-segment of a word $w$ we have $\sum_{i=1}^{r} \mathrm{dp}(w_{i})-1 = \mathrm{dp}(w) - r$, we deduce : for all words $e_{0}^{n_{d}-1}e_{1}\cdots e_{0}^{n_{1}-1}e_{1}$ and all $m\geq1$, we have :

\begin{multline} \label{eq:3.5}
(\Phi_{p}^{\De}  \text{ } \newaction \text{ }  (G_{0}^{\alpha, \reg},\Li_{1}^{(p)}))_{e_{0}^{n_{d}-1}e_{1}\cdots e_{0}^{n_{1}-1}e_{1}}[z^m] = 
\\ \sum_{d'=0}^{d}\sum_{\substack{a,b\geq 0 \\ a+b \leq n_{d'}-1}}
\sum_{\substack{0\leq i,j,l \\ i+l+pj = m}} \bigg( h^{B}_{e_0^{n_{d}-1}e_1 \dots e_0^{n_{d'+1}-1}e_1e_{0}^{a}}(i)  
 \cdot
\sum_{\substack{i_1,\ldots,i_b \geq 0 \\ i_1+\cdots+i_b=l \\ (i_1,p)=\cdots=(i_b,p)=1}}\frac{1}{b!\cdot i_1\cdots i_b} \cdot
\\
\
\sum_{\substack{1\leq r \\ (w_{k})_{k=1}^{r}  \text{$e_1$-segments of }\\ (e_0^{n_{d'}-1-a-b} e_1 \dots e_0^{n_{1}- 1}e_1)^\rev}} 
\frac{(-1)^{\sum_{i=1}^{d'} n_{i} - a-b}}{p^{\sum_{i=1}^{d'}n_{i}-a-b-\sum_{k=1}^r(\wt(w_k)-1)}}
\bigg((-1)^{d'-r}\prod_{k=1}^{r} (\zeta_{p}^{\De})^{\Ad}(w_{k}) \bigg)\cdot h^{B}_{(e_0^{n_{d'}-1-a-b} e_1 \dots e_0^{n_{1}- 1}e_1)^{rev} / (w_{k})_k} (j) \bigg) .
\end{multline}
Here we put 
\[\sum\limits_{\substack{i_1,\ldots,i_b \geq 0 \\ i_1+\cdots+i_b=l \\ (i_1,p)=\cdots=(i_b,p)=1}}\frac{1}{b!\cdot i_1\cdots i_b}=
\begin{cases}
    1 & \text{when } b=0,\  l=0, \\
    0 & \text{when }b=0,\ l> 0.
\end{cases}
\]

Applying Proposition \ref{prop sec 1}, Theorem \ref{prop Mahler} to equation \eqref{eq:3.5}, and the fact that $\zeta_{p}^{\De}(w) = (-1)^{\mathrm{dp}(w)} (\Phi_{p}^{\De})_{w}$ for all words $w$, we deduce Theorem \ref{thm 0.2}.
\qed
\newline 
\newline 
The formula of Theorem \ref{thm 0.2} is inductive with respect to the depth $d$ of indices $(n_{1},\ldots,n_{d})$ of $p$-adic multiple zeta values. Indeed, in equation (\ref{eq:3.5}), each $w_{k}$ is a subword of $w=(n_{1},\ldots,n_{d})$ thus 
$\dep w_{k} \leq \dep w$. By equation \eqref{eq:adjoint pMZVs}, each $(\zeta^{\De}_{p})^{\Ad}(w_{k})$ is thus a polynomial of $p$-adic multiple zeta values of depth $\leq \dep(w) -1$ (indeed, note that in equation (\eqref{eq:adjoint pMZVs}), $w'$ is of depth $d+1$ and the indices of $p$-adic multiple zeta values in the right-hand side are of depth $\leq d$.)

Let us also recall that, for all $n\geq 1$, $(\zeta_{p}^{\De})^{\Ad}(0;n;0)  =  (1+(-1)^{n}) \zeta_{p}^{\De}(n_{1})=0$, since it is known that $\zeta_{p}^{\De}(n)=0$ for $n$ even. 

\bigskip
Now we explain how we recover the formulas of Example \ref{ex depth1} and Example \ref{ex depth2} as examples of the general formula of Theorem \ref{thm 0.2}. 

\bigskip
{\bf Proof  of Example \ref{ex depth1}} :

For $d=1$, the sum over $d'$ in Theorem \ref{thm 0.2} has two terms, multiplied by $(-1)^{d}=1$ : 

$\bullet$ $d'=0$ : this term is $h^{B}_{e_{0}^{n-1}e_1}(m)$.

$\bullet$ $d'=1$ : since $h^{B}_{e_{0}^{a}}(i) = 0$ for all $a\geq 0$ and $i\geq 0$ except when $a=i=0$ in which case it is 1, the sums over $a$ and $i$ can be restricted to $a=i=0$, with the factor $h^{B}_{e_{0}^{a}}(i)=h^B_1(0)=1$. The $e_{1}$-segments of $(e_{0}^{n-1-b}e_{1})^{\rev} = e_{1}e_{0}^{n-1-b}$ are $(e_{1}e_{0}^{n-1-b-k})$, $0\leq k \leq n-1-b$, and the contraction of $(e_{0}^{n-1-b}e_{1})^\rev$ by $(e_{1}e_{0}^{n-1-b-k})$ is $e_{1}e_{0}^{k}$. Finally, since $(\zeta_{p}^{\De})^{\Ad}(e_{1}e_{0}^{n-1-b-k}) = ((\Phi_{p}^{\De})^{-1}e_{1}\Phi_{p}^{\De})_{e_{1}e_{0}^{n-1-b-k}} = \Phi_{p}^{De}[e_{0}^{n-1-b-k}] = \left\{ \begin{array}{ll} 1 & \text{ if } k=n-1-b \\ 0 & \text{if } k<n-1-b \end{array} \right.$. Thus, only the $k=n-1-b$ term remains, with a factor $(\zeta_{p}^{\De})^{\Ad}(e_{1})=1$. Thus we obtain the term

$$ \sum_{0\leq b \leq n-1} \sum_{\substack{l,j\geq 0 \\ l+pj=m}} (\frac{-1}{p})^{n-b} \sum_{\substack{i_{1}+\cdots+i_{b}=l \\ \forall j,  p\nmid i_{j} }} \frac{1}{b!\cdot i_{1}\cdots i_{b}}  h^{B}_{e_{1}e_{0}^{n-1-b}}(j) . $$
The result is the sum of the terms $d'=0$ and $d'=1$, multiplied by $(-1)^{d}=-1$.
\qed

\bigskip
{\bf Proof of Example \ref{ex depth2}}:

For $d=2$, the sum over $d'$ in Theorem \ref{thm 0.2} has three terms : 

$\bullet$ $d'=0$ : this term is $h^{B}_{e_{0}^{n_{2}-1}e_{1}e_{0}^{n_{1}-1}e_1}(m)$.

$\bullet$ $d'=1$ : this is similar to the $d'=1$ case of depth $1$, except that this time we have to consider the factors $h^{B}_{e_{0}^{n_{2}-1}e_{1}e_{0}^{a}}(i)$ which do not vanish, thus we obtain :
$$ \sum_{\substack{0\leq a,b \\a+b \leq n_{1}-1}} \sum_{\substack{i,l,j\geq 0 \\ i+l+pj=m}} (\frac{-1}{p})^{n_{1}-b} h^{B}_{e_{0}^{n_{2}-1}e_{1}e_{0}^{a}}(i) \sum_{\substack{i_{1}+\cdots+i_{b}=l \\ \forall j,  p\nmid i_{j} }} \frac{1}{b!\cdot i_{1}\cdots i_{b}}  h^{B}_{e_{0}^{n-1-b}e_{1}}(j) . $$

$\bullet$ $d'=2$ : as in the $d'=1$ case of depth $1$, we have $h^{B}_{e_{0}^{a}}(i)=0$ except when $a=i=0$ where this is $1$, thus we can restrict the sums over $a$ and $i$ to $a=0$ and $i=0$. Thus we have to consider the $e_{1}$-segments of the words $(e_{0}^{n_{2}-1-b}e_{1}e_{0}^{n_{1}-1}e_{1})^{\rev}=e_{1}e_{0}^{n_{1}-1}e_{1}e_{0}^{n_{2}-1-b}$, for $b\geq 0$ such that $b \leq n_{2}-1$. They are of two possible types :

(i) $(e_{1}e_{0}^{n_{1}-1}e_{1}e_{0}^{n_{2}-1-b-k})$, $0\leq k \leq n_{2}-1-b$. The corresponding contraction of $e_{1}e_{0}^{n_{1}-1}e_{1}e_{0}^{n_{2}-1-b}$ is $e_{1}e_{0}^{k}$. 
By \eqref{eq:adjoint pMZVs},
we have 
$$\zeta_{p}^{\Ad}(e_{1}e_{0}^{n_{1}-1}e_{1}e_{0}^{n_{2}-1-b-k}) =  \left\{ 
\begin{array}{ll} 0 & \text{ if } k=n_{2}-1-b \\ {-n_{1} \choose n_{2}-1-b-k} \zeta_{p}(n_{1}+n_{2}-1-b-k) & \text{if } k<n_{2}-1-b  \\
\end{array} 
\right. . $$ 

Thus we obtain the term 

\begin{multline*} -\sum_{\substack{b\geq 0 \\ b \leq n_{d'}-1}}
\sum_{\substack{0\leq j,l \\ l+pj = m}} \bigg( 
\sum_{\substack{i_1+\cdots+i_b=l \\ (i_1,p)=\cdots=(i_b,p)=1}}\frac{1}{b!\cdot i_1\cdots i_b} \cdot
\sum_{k=0}^{n_{2}-2-b}
\frac{(-1)^{n_1+n_2-b}}{p^{k+1}}
{-n_{1} \choose n_{2}-1-b-k} \zeta_{p}(n_{1}+n_{2}-1-b-k)
 h^{B}_{e_{1}e_{0}^{k}} (j) \bigg).
\end{multline*}

(ii) $(e_{1}e_{0}^{n_{1}-1-k''},e_{0}^{k'}e_{1}e_{0}^{n_{2}-1-b-k})$
with $0 \leq k',k''\leq n_{1}-1$ such that $k'\leq k''$ and $0\leq k \leq n_{2}-1-b$. The corresponding contraction of 
$(e_{0}^{n_{2}-1-b}e_{1}e_{0}^{n_{1}-1}e_{1})^{\rev}= e_{1}e_{0}^{n_{1}-1}e_{1}e_{0}^{n_{2}-1-b}$ is $e_{1}e_{0}^{k''-k'}e_{1}e_{0}^{k}$.

We have $\zeta_{p}^{\Ad}(e_{1}e_{0}^{n_{1}-1-k''}) 
= \left\{ 
\begin{array}{ll} 
1 & \text{if }k''=n_{1}-1,
\\0  & \text{if }k''<n_{1}-1
\end{array}
\right.$,
$\zeta_{p}^{\Ad}(e_{0}^{k'}e_{1}e_{0}^{n_{2}-1-b-k}) 
= \left\{ 
\begin{array}{ll} 
1 & \text{if }k'=0,k=n_{2}-1-b,
\\ 0  & \text{if }k'>0 \text{ or }k<n_{2}-1-b
\end{array}
\right. 
$.

Thus we obtain the term 

\begin{equation*} 
\sum_{\substack{b\geq 0 \\ b \leq n_{2}-1}}
\sum_{\substack{0\leq j,l \\ l+pj = m}} \bigg( 
\sum_{\substack{i_1+\cdots+i_b=l \\ (i_1,p)=\cdots=(i_b,p)=1}}\frac{1}{b!\cdot i_1\cdots i_b} \cdot
\frac{(-1)^{n_{1}}}{p^{n_{2}-1}} 
\cdot h^{B}_{e_{1}e_{0}^{n_{1}-1}e_{1}e_{0}^{n_{2}-1-b}}\bigg) . 
\end{equation*}

\qed

We finish this paper by a few comments.

\begin{rem} A slightly more general version of the main theorem could be stated and proved similarly, by considering not only the indices $(n_{1},\ldots,n_{d})=e_{0}^{n_{d}-1}e_{1}\cdots e_{0}^{n_{1}-1}e_{1}$ but more generally the indices $e_{0}^{n_{d}-1}e_{1}\cdots e_{0}^{n_{1}-1}e_{1}e_{0}^{n_{0}-1}$, with $n_{0}\in \mathbb{N}_{\geq 1}$. Here, we have restricted the formula to $n_{0}=1$. This is sufficient because, by the shuffle equation and the fact that $\Phi_{p}^{\De}[e_{0}]=0$, we have 
$$(\Phi_{p}^{\De}[e_{0}^{n_{d}-1}e_{1}\cdots e_{0}^{n_{1}-1}e_{1}e_{0}^{n_{0}-1}] = \sum\limits_{\substack{k_{1},\ldots,k_{d}\geq 0 \\ k_{1}+\cdots+k_{d}=n_{0}-1}} \prod_{i=1}^{d} {-n_{i} \choose k_{i}} \zeta_{p}^{\De}(n_{1}+k_{1},\ldots,n_{d}+k_{d}).$$
\end{rem}

\begin{rem}
A generalization of Deligne's $p$-adic multiple zeta values was introduced by the second author \cite{J1} : for all $\alpha \in \mathbb{Z} \setminus \{0\}$ we have an element  $\Phi_{p,\alpha} \in \mathbb{Q}_{p} \langle\langle e_0,e_1\rangle\rangle$ and the numbers $\zeta_{p,\alpha}(n_{1},\ldots,n_{d}) = (-1)^{d}(\Phi_{p,a})_{e_{0}^{n_{d}-1}e_{1} \cdots e_{0}^{n_{1}-1}e_{1}}$ \cite[Definition 1.2.2]{J1}. The $\alpha=1$ case is equivalent to the case treated in this paper. Acutally, our proof works similarly for any $\alpha$. If $\alpha > 0$, the formulas similar with $p$ replaced by $p^{\alpha}$. If $\alpha<0$, we have instead of Equation \eqref{eq:overconvergent polylog} a different but similar equation for overconvergent $p$-adic multiple polylogarithms (see for example \cite[Equation (5.3)]{J3}, and the proof and formulas are still similar.
\end{rem}

\begin{rem} Usually, computations of $p$-adic multiple zeta values have a direct generalization to $p$-adic cylotomic multiple zeta values. Here, however, the automorphism $z \mapsto \frac{z}{z-1}$ of $\mathbb{P}^{1} \setminus \{0,1,\infty\}$ plays a central role, and it does not stabilize $\mathbb{P}^{1} \setminus \{0,\mu_{N},\infty\}$. Thus, there is no direct generalization of our results to the cyclotomic case. Instead, a generalization in a larger sense would involve other operations on the unipotent fundamental group of $\mathbb{P}^{1} \setminus \{0,\mu_{N},\infty\}$.
\end{rem}


\end{document}